\documentclass{amsart}
\usepackage{amsmath, amscd, amssymb, amsthm}
\usepackage{bbm}
\usepackage{latexsym}
\usepackage{amsfonts}
\usepackage{graphicx}
\usepackage[all,cmtip]{xy}

\usepackage[
colorlinks=true, linkcolor=black,breaklinks=true,
urlcolor=blue, citecolor=black]{hyperref}%
\usepackage{geometry}
\newtheorem{theorem}{Theorem}
\newtheorem{lemma}{Lemma}

\newtheorem{conjecture}{Conjecture}

\newcommand{\be}{\begin{equation}}
\newcommand{\ee}{\end{equation}}
\newcommand{\bea}{\begin{eqnarray}}
\newcommand{\eea}{\end{eqnarray}}

\def\XXint#1#2#3{{\setbox0=\hbox{$#1{#2#3}{\int}$ }
\vcenter{\hbox{$#2#3$ }}\kern-.6\wd0}}

\begin{document}

\title[Pluriclosed manifolds with constant $H$]{Pluriclosed manifolds with constant holomorphic sectional curvature}


\author{Peipei Rao}
\address{Peipei Rao. School of Mathematical Sciences, Chongqing Normal University, Chongqing 401331, China}
\email{{1291012563@qq.com}}

\author{Fangyang Zheng} \thanks{The research is partially supported by NSFC grant \# 12071050 and Chongqing Normal University.}
\address{Fangyang Zheng. School of Mathematical Sciences, Chongqing Normal University, Chongqing 401331, China}
\email{{20190045@cqnu.edu.cn}}


\begin{abstract}
A long-standing conjecture in complex geometry says that a compact Hermitian manifold with constant holomorphic sectional curvature must be K\"ahler when the constant is non-zero and must be Chern flat when the constant is zero. The conjecture is known in complex dimension $2$ by the work of Balas-Gauduchon in 1985 (when the constant is zero or negative) and by Apostolov-Davidov-Muskarov in 1996 (when the constant is positive). For higher dimensions, the conjecture is still largely unknown. In this article, we restrict ourselves to pluriclosed manifolds, and confirm the conjecture for the special case of Strominger K\"ahler-like manifolds, namely, for Hermitian manifolds whose Strominger connection (also known as Bismut connection) obeys all the K\"ahler symmetries.
 \end{abstract}

\maketitle


\markleft{Rao and Zheng}
\markright{Pluriclosed manifolds with constant $H$}

\section{Introduction and statement of result}

Let us denote by $R$ the curvature tensor of the Chern connection of a given Hermitian manifold $(M^n,g)$. The holomorphic sectional curvature $H$ is defined by
$$ H(X) = R_{X\overline{X}X\overline{X}} / |X|^4 $$
where $X\neq 0$ is any type $(1,0)$ tangent vector on $M$. When the metric $g$ is K\"ahler, it is well known that the values of $H$ determines the entire $R$, and complete K\"ahler manifolds with constant $H$ are exactly the complex space forms, namely, with universal covering space being ${\mathbb C}{\mathbb P}^n$, ${\mathbb C}^n$, or ${\mathbb C}{\mathbb H}^n$ equipped with (scaling of) the standard metric. When $g$ is non-K\"ahler, $R$ does not obey the usual symmetry conditions as in the K\"ahler case, and the values of $H$ do not determine the entire curvature tensor $R$. Nonetheless, the following long-standing folklore conjecture is still believed by many when the manifold is compact:


\begin{conjecture} \label{conjecture1}
Let $(M^n,g)$ be a compact Hermitian manifold with  $H$  equal to a constant $c$. Then $g$ is K\"ahler if $c\neq 0$ and $g$ is Chern flat (namely, $R=0$) if $c=0$.
\end{conjecture}


Note that compact Chern flat manifolds have been classified by Boothby \cite{Boothby} in 1958 as all the compact quotients of complex Lie groups, equipped with left invariant metrics.

 The conjecture is known to be true in dimension two. In 1985,  Balas and Gauduchon proved in \cite{BG} that when $c\leq 0$ the conjecture holds (see also \cite{Gauduchon1} and \cite{Balas} for earlier work). In 1996, Apostolov, Davidov and Muskarov \cite{ADM} solved the remaining $c>0$ case in dimension two.

 For dimension three or higher, the conjecture is still largely open, with only a few partial results known. Kai Tang in \cite{Tang} proved the conjecture under the additional assumption that the metric is Chern K\"ahler-like, meaning that the curvature tensor $R$ obeys all the K\"ahler symmetries. In their recent paper \cite{CCN}, Chen-Chen-Nie proved the conjecture for the case $c\leq 0$ under the additional assumption that $g$ is locally conformally K\"ahler. They also pointed out the necessity of the compactness assumption in the conjecture by explicit examples.

 For $n\geq 3$, Conjecture 1 seems to be a daunting task at this moment, and some people are actually hoping for counterexamples which would certainly form a very interesting class of non-K\"ahler manifolds if exist. For this reason, perhaps a less ambitious and more  realistic goal might be to restrict to some special classes of compact Hermitian manifolds. It seems to us that there are at least the following three ways to approach it.

 The first is consider {\em locally homogeneous Hermitian manifolds,} namely, a compact Hermitian manifold $(M^n,g)$ whose universal covering space is a homogeneous Hermitian manifold. A large and important subset of this is the so-called {\em Lie-Hermitian manifolds,} which means compact Hermitian manifold $(M^n,g)$ whose universal covering space is a (connected, simply-connected, even-dimensional) Lie group $G$ equipped with a left invariant complex structure and a compatible left invariant metric. In a recent work \cite{LZ}, Y. Li and the second named author confirmed the conjecture for complex nilmanifolds, namely, Lie-Hermitian manifolds with $G$ nilpotent. For reasons explained in \cite{LZ}, we believe that one should explore other classes of Lie-Hermitian manifolds, or more generally locally homogeneous Hermitian manifolds, in hope of either proving the conjecture in the special case or producing counterexamples.

 The second special class to consider would be the set of {\em balanced } manifolds, namely, Hermitian manifold $(M^n,g)$ with $d(\omega^{n-1})=0$. Here $\omega$ is the K\"ahler form. Such manifolds form an important subset of non-K\"ahler manifolds even for $n=3$. For instance, it includes all the twistor spaces and many known examples of non-K\"ahler Calabi-Yau spaces.

 In a recent work \cite{ZZ}, W. Zhou and the second named author proved the following statement: any compact Hermitian threefold with vanishing real bisectional curvature must be Chern flat. This is a special case of balanced threefolds. Recall that {\em real bisectional curvature} is a curvature notion on Hermitian manifolds introduced by X. Yang and the second named author in \cite{YangZ}. It is equivalent to $H$ in strength when the metric is K\"ahler, but slightly stronger than $H$ in the non-K\"ahler case, so the main result of \cite{ZZ} is weaker than the balanced case of Conjecture 1 for $n=3$ and $c=0$.

 The third special class for Conjecture 1 is to consider {\em pluriclosed} manifolds, namely, a Hermitian manifold $(M^n,g)$ such that $\partial \overline{\partial } \omega =0$. A well-known conjecture in complex geometry states that if a compact complex manifold admits a balanced Hermitian metric and a pluriclosed metric, then it is K\"ahlerian, namely, it admits a K\"ahler metric. So in this spirit the second and third special classes are mutually exclusive.

 The main purpose of this article is examine the conjecture for pluriclosed manifolds. In the next section, we will analyze the geometric properties of pluriclosed manifolds with constant holomorphic sectional curvature. The result is summarised as Theorem 2 there. While we cannot establish the conjecture for all pluriclosed manifolds at this point, we will prove the following special case of it:

 \begin{theorem}
 Let $(M^n,g)$ be a compact Hermitian manifold whose holomorphic sectional curvature $H$ is equal to a constant $c$. If $g$ is Strominger K\"ahler-like, then $g$ must be K\"ahler thus $(M^n,g)$ is a complex space form.
 \end{theorem}

Recall that a metric connection $D$ on a Hermitian manifold $(M^n,g)$ is said to be {\em K\"ahler-like,} if its curvature tensor $R^D$ obeys all the K\"ahler symmetries, namely, for any type $(1,0)$ complex tangent vectors $X$, $Y$, $Z$, $W$, the only possibly non-zero components of $R^D$ are $R^D_{X\overline{Y}Z\overline{W}}$, and
$\, R^D_{X\overline{Y}Z\overline{W}} = R^D_{Z\overline{Y}X\overline{W}}$.

This notion was introduced in \cite{YZ} for Levi-Civita connection and  Chern connection, following pioneer work of Gray and others, and it was introduced and studied for all metric connections by Angella, Otal, Ugarte and Villacampa in \cite{AOUV}.

Given a Hermitian manifold $(M^n,g)$, the Strominger connecton $\nabla^s$ (also known as Bismut connection) is the unique connection that is Hermitian (i.e., $\nabla^s g=0$, $\nabla^s J=0$) and its torsion tensor is totally skew-symmetric. This is an important canonical connection for Hermitian manifolds, which serves as a bridge between Levi-Civita connection and Chern connection. When $\nabla^s$ is K\"ahler-like, the metric is said to be {\em Strominger K\"ahler-like,} or $SKL$ for short.

It was conjectured in \cite{AOUV} that all compact SKL manifolds are pluriclosed. This was confirmed in \cite{ZhaoZ1} by Q. Zhao and the second named author. In \cite{YZZ}, more properties of SKL manifolds were analyzed and a classification theorem was established for such manifolds in dimension $3$.

Note that SKL manifolds includes all {\em Strominger flat} manifolds, which were classified by Q. Wang, B. Yang and the second named author in \cite{WYZ} as quotients of Samelson spaces, namely, Lie-Hermitian manifolds with bi-invariant metrics. More precisely, by Milnor's Lemma \cite{Milnor}, the universal cover is the product of a compact semi-simple Lie group with a vector group, equipped with a bi-invariant metric and a compatible left invariant complex structure.

Since any compact Chern flat manifold is always balanced, while SKL manifolds are pluriclosed, we know that in Theorem 1 the metric has to be K\"ahler in the $c=0$ case as well, by the well-known fact that any Hermitian metric that is simultaneously balanced and pluriclosed must be K\"ahler.

The article is organized as follows. In the next section, we will discuss general properties of pluriclosed manifolds with constant holomorphic sectional curvature, and summarize the results as Theorem 2. In the third section, we will specialize to SKL manifolds and prove Theorem 1.

\vspace{0.3cm}

\section{Pluriclosed manifolds with constant holomorphic sectional curvature }

First let us set up the notations.  Let $(M^n,g)$ be a Hermitian manifold, and let $J$ be the almost complex structure associated with the complex structure of $M$. Denote by $\nabla$ the Chern connection, which is the unique connection that is Hermitian (i.e., $\nabla g=0$ and $\nabla J=0$) and satisfies $(\nabla )^{0,1}=\overline{\partial}$. Denote by $T$, $R$ the torsion and curvature tensor of $\nabla$, namely,
$$ T(x,y) = \nabla_xy - \nabla_yx -[x,y], \ \ \ \ \ R_{xy}z=\nabla_x\nabla_y z - \nabla_y\nabla_x z -\nabla_{[x,y]} z$$
where $x$, $y$ and $z$ are tangent vectors on $M$. Write $g=\langle \, ,\, \rangle $ and $R_{xyzw}=\langle R_{xy}z, w\rangle$. Extend $g$, $T$, and $R$ linearly over ${\mathbb C}$. It is well-known that
$$T(X,\overline{Y})=0, \ \ \ \ R_{XY\ast \ast } = R_{\ast \ast ZW}=0$$
for any type $(1,0)$ complex tangent vectors $X$, $Y$, $Z$  and $W$, so the only possibly non-zero components of $R$ are $R_{X\overline{Y}Z\overline{W}}$. Suppose $\{ e_1, \ldots , e_n\}$ is a local unitary frame of type $(1,0)$ tangent vectors. Under the frame $e$, let us write
\begin{equation}
T(e_i,e_j) = \sum_{k=1}^n T^k_{ij} e_k
\end{equation}
where $T^k_{ij} = -T^k_{ji}$. Note that our $T^k_{ij}$ is twice of that in \cite{YZ}. Also write $R_{i\overline{j}k\overline{\ell}}$ for $R_{e_i\overline{e}_je_k\overline{e}_{\ell}}$.  By definition, the holomorphic sectional curvature $H$ of $R$ is equal to a constant $c$ if and only if
$$ R_{X\overline{X}X\overline{X}} = c |X|^4  $$
for any type $(1,0)$ tangent vector $X$. Under the unitary frame $e=\{ e_1, \ldots , e_n\}$, this means
\begin{equation}
 \widehat{R}_{i\overline{j}k\overline{\ell}}=\frac{c}{2}(\delta_{ij}\delta_{k\ell} + \delta_{i\ell}\delta_{kj}), \label{Rhat}
\end{equation}
where
\begin{equation}
\widehat{R}_{i\overline{j}k\overline{\ell}} = \frac{1}{4} \big( R_{i\overline{j}k\overline{\ell}} + R_{k\overline{j}i\overline{\ell}} +  R_{i\overline{\ell}k\overline{j}} +  R_{k\overline{\ell}i\overline{j}} \big)
\end{equation}
is the symmetrization of $R$. Note that when the metric $g$ is K\"ahler, the well-known K\"ahler symmetry says that $\widehat{R}=R$, so $H$ determines the entire $R$. However, for a general Hermitian metric, $\widehat{R}$ might not be equal to $R$, in which case $H$ does not determine $R$. This is where the difficulty lies in proving Conjecture 1.

Next, let $\{ \varphi_1, \ldots , \varphi_n\}$ be the unitary coframe dual to $e$, and denote by $\theta$, $\Theta$ the matrix of connection and curvature of the Chern connection under $e$. Namely,
$ \nabla e_i = \sum_{j=1}^n \theta_{ij}e_j$. The structure equations are
$$ d\varphi_i = - \sum_{j=1}^n \theta_{ji} \wedge \varphi_j + \tau_i , \ \ \ \ \ \Theta_{ij} = d\theta_{ij} - \sum_{r=1}^n \theta_{ir}\wedge \theta_{rj} $$
where
$$ \tau_i = \sum_{j<k} T^i_{jk} \,\varphi_j\wedge \varphi_k = \frac{1}{2} \sum_{j,k=1}^n T^i_{jk} \, \varphi_j\wedge \varphi_k $$
is the column vector of torsion forms, and
$$ \Theta_{ij} = \sum_{k,\ell=1}^n R_{k\overline{\ell}i\overline{j}}\,  \varphi_k \wedge \overline{\varphi}_{\ell}. $$
As a direct consequence of the structure equations, we have
\begin{equation*}
\partial \overline{\partial} \omega = - \sqrt{-1} \{ \,^t\!\varphi \wedge \Theta \wedge \overline{\varphi} + \,^t\!\tau \wedge \overline{\tau}\},
\end{equation*}
where $\omega = \sqrt{-1}\,^t\!\varphi \wedge \overline{\varphi}$ is the K\"ahler form, and $\varphi$, $\tau$ are understood as column vectors of $(1,0)$ or $(2,0)$ forms, respectively. Write them in components, we get

\begin{lemma}
A Hermitian manifold $(M^n,g)$ is pluricolsed if and only if under any local unitary frame $e$ it holds
\begin{equation*}
\big( R_{i\overline{j}k\overline{\ell}} + R_{k\overline{\ell}i\overline{j}}\big)  -  \big( R_{k\overline{j}i\overline{\ell}} + R_{i\overline{\ell}k\overline{j}}\big) = \sum_{r=1}^n T^r_{ik} \overline{T^r_{j\ell}}
\end{equation*}
for any $1\leq i,j,k,\ell \leq n$.
\end{lemma}

By the first Bianchi identity, we get the following identity between curvature $R$ and the covariant derivative of $T$ with respect to the Chern connection $\nabla$:

\begin{lemma}
For any Hermitian manifold $(M^n,g)$ and under any local unitary frame $e$, it holds
\begin{equation*}
 R_{k\overline{j}i\overline{\ell}} - R_{i\overline{j}k\overline{\ell}} =  T^{\ell}_{ik,\,\overline{j}}
\end{equation*}
for any $1\leq i,j,k,\ell \leq n$, where the index after comma stands for covariant derivative with respect to the Chern connection $\nabla$.
\end{lemma}

This is just Formula (21) in \cite[Lemma 7]{YZ}. Note that our $T^k_{ij}$ is twice of the same notation in \cite{YZ}, and our $R$ is denoted as  $R^h$ there.

Next let us recall Gauduchon's torsion $1$-form $\eta$ \cite{Gauduchon} (here we took its $(1,0)$ part) which is defined by $\partial (\omega^{n-1}) = - \eta \wedge \omega^{n-1}$. It is a global $(1,0)$ form on $M^n$, and under any local unitary frame $e$, it has the expression:
$$ \eta = \sum_{i=1}^n \eta_i \varphi_i, \ \ \ \ \mbox{where} \ \ \ \eta_i = \sum_{k=1}^n T^k_{ki}. $$
Denote by $\chi = \sum_{i=1}^n \eta_{i, \overline{i}}$ the global smooth function on $M^n$ where $e$ is unitary and index after comma stands for covariant derivative with respect to the Chern connection $\nabla$. A direct computation leads to

\begin{lemma}
On any Hermitian manifold $(M^n,g)$, it holds
\begin{equation*}
 n\sqrt{-1} \partial \overline{\partial } (\omega^{n-1}) = (|\eta|^2-\chi  )\omega^n.
\end{equation*}
In particular, $\overline{\chi }=\chi$, and when $M^n$ is compact, we always have $\,\int_M \chi = \int_M |\eta|^2$.
\end{lemma}

Here $|\eta|^2$ stands for $\sum_{i=1}^n |\eta_i|^2$ under any unitary frame. Similarly, we will denote by $|T|^2$ the quantity $\sum_{i,j,k=1}^n |T^j_{ik}|^2$ under any unitary frame. Recall that  the {\em first, second, and third Ricci form} of the Chern curvature tensor $R$ are $(1,1)$-forms on $M^n$ defined by
$$\rho^{(r)} = \sqrt{-1} \sum_{i,j=1}^n \rho^{(r)}_{i\overline{j}} \varphi_i \wedge \overline{\varphi}_j, \ \ \ \ \ \ \ r=1, 2, 3,   \ \ \ \  \mbox{where} \ \ $$
\begin{equation}
 \rho^{(1)}_{i\overline{j}}=\sum_{r=1}^n R_{i\overline{j}r\overline{r}}, \ \ \ \ \ \rho^{(2)}_{i\overline{j}}=\sum_{r=1}^n R_{r\overline{r}i\overline{j}}, \ \ \ \ \ \rho^{(3)}_{i\overline{j}}=\sum_{r=1}^n R_{r\overline{j}i\overline{r}}.
 \end{equation}
There are two scalar curvatures associated with the Chern curvature $R$, namely, the trace of $ \rho^{(1)}$ or equivalently $ \rho^{(2)}$, denoted as $s$, and the trace of $ \rho^{(3)}$, denoted as $\hat{s}$:
\begin{equation}
s = \sum_{i,k=1}^n R_{i\overline{i}k\overline{k}}, \ \ \ \ \ \hat{s} = \sum_{i,k=1}^n R_{k\overline{i}i\overline{k}}.
\end{equation}
By Lemma 2, if we let $i=j$ and $k=\ell$ and sum up, then we immediately get the following
\begin{equation}
s - \hat{s} = \chi. \label{eq:sminuss}
\end{equation}

\begin{lemma}
Let $(M^n,g)$ be a Hermitian manifold. Under any local unitary frame $e$, it holds
\begin{equation}
 \sum_{i,j=1}^n \eta_{i, \overline{j}} \, \varphi_i \wedge \overline{\varphi}_j  = - \overline{\partial }\eta. \label{eq:dbareta}
\end{equation}
Here the index after comma again stands for covariant derivative with respect to the Chern connection.   In particular, Lemma 2 implies that
\begin{equation}
 \rho^{(3)} - \rho^{(1)} = \sqrt{-1} \,\overline{\partial }\eta . \label{eq:rhominusrho}
\end{equation}
\end{lemma}

\begin{proof}
Note that the left hand side of (\ref{eq:dbareta}) is independent of the choice of the unitary frame, hence is a globally defined $(1,1)$-form on $M^n$. Fix any point $p$ in $M$. We may choose a local unitary frame $e$ in a neighborhood of $p$ so that the  matrix $\theta$ of the Chern connection $\nabla$ under $e$ vanishes at $p$ (see for example \cite[Lemma 4]{YZ}. Use this frame, then at the point $p$ we have $\overline{\partial }\varphi =0$, hence
$$ \overline{\partial }\eta = \overline{\partial }\eta_i \wedge \varphi_i = \overline{e}_j (\eta_i) \,\overline{\varphi}_j \wedge \varphi_i = \eta_{i,\,\overline{j}}\, \overline{\varphi}_j \wedge \varphi_i.$$
This proves (\ref{eq:dbareta}).  Now let $k=\ell$ in Lemma 2 and sum up, we get (\ref{eq:rhominusrho}).
\end{proof}

Note that if we take trace on both sides of (\ref{eq:rhominusrho}), and use (\ref{eq:dbareta}), we get $\hat{s}-s = -\chi$, which is just (\ref{eq:sminuss}). Let us also introduce the following two $(1,1)$-forms:
\begin{eqnarray}
 \xi & = & \sqrt{-1} \sum_{i,j=1}^n \xi_{i\overline{j}} \, \varphi_i \wedge \overline{\varphi}_j , \ \ \ \ \ \ \xi_{i\overline{j}} \, = \, \sum_{r=1}^n T^j_{ir, \, \overline{r}}; \\
 \sigma & = & \sqrt{-1} \sum_{i,j=1}^n \sigma_{i\overline{j}} \, \varphi_i \wedge \overline{\varphi}_j , \ \ \ \ \ \sigma_{i\overline{j}} \, = \, \sum_{r,s=1}^n T^r_{is} \overline{ T^r_{js}}
 \end{eqnarray}
Clearly both of them are independent of the choice of local unitary frames, thus are globally defined $(1,1)$-forms on $M^n$. Also, $\sigma = \overline{\sigma }\geq 0$.

Now we are ready to state the property for pluriclosed manifold with constant holomorphic sectional curvature:

\begin{theorem}
Let $(M^n,g)$ be a pluriclosed manifold with holomorphic sectional curvature $H$ equal to a constant $c$. Then under any local unitary frame $e$, it holds
\begin{equation}
T^j_{ik,\,\overline{\ell}} - T^{\ell}_{ik,\,\overline{j}}  = \sum_{r=1}^n T^r_{ik}  \overline{ T^r_{j\ell} } \label{eq:TT}
\end{equation}
and the Chern curvature tensor is given by
\begin{equation}
R_{i\overline{j}k\overline{\ell}} = \frac{c}{2}\big( \delta_{ij} \delta_{k\ell} + \delta_{i\ell } \delta_{kj}\big) - \frac{1}{4} \sum_{r=1}^n T^r_{ik}  \overline{ T^r_{j\ell} } - \frac{1}{2} \big( T^{\ell}_{ik,\,\overline{j}} + \overline{ T^{k}_{j\ell ,\,\overline{i}} } \big)  \label{eq:R}
\end{equation}
for any $1\leq i,j,k,\ell \leq n$. By taking various traces in $(\ref{eq:TT})$ and $(\ref{eq:R})$, we get
\begin{eqnarray}
&& \xi_{i\overline{j}} + \eta_{i,\overline{j}} \,= \,\sigma_{i\overline{j}} , \ \ \ \ \mbox{hence} \ \ \ \ \xi = \sqrt{-1}\,\overline{\partial} \eta + \sigma  \ \ \ \ \mbox{and} \ \ \ 2\chi \, = \, |T|^2, \\
&& \rho^{(1)} \, = \, \frac{c}{2}(n+1)\omega - \frac{1}{4} \sigma + \frac{\sqrt{-1}}{2} (\partial \overline{\eta} - \overline{\partial }\eta ) \\
&& \rho^{(2)} \, = \, \frac{c}{2}(n+1)\omega + \frac{3}{4} \sigma - \frac{\sqrt{-1}}{2} (\partial \overline{\eta} - \overline{\partial }\eta ) \\
&& \rho^{(3)} \, = \, \frac{c}{2}(n+1)\omega - \frac{1}{4} \sigma + \frac{\sqrt{-1}}{2} (\partial \overline{\eta} + \overline{\partial }\eta )
\end{eqnarray}
In particular, when $M^n$ is compact it holds that $\ 2\int_M|\eta|^2 = \int_M|T|^2$.
\end{theorem}

\begin{proof}
We start from Lemma 2. Switch $j$ and $\ell$ in Lemma 2, we get
$$ R_{k\overline{\ell}i\overline{j}} - R_{i\overline{\ell}k\overline{j}} = T^{j}_{ik, \,\overline{\ell}}. $$
Subtract from that the original formula of Lemma 2, we get
$$ T^{j}_{ik, \,\overline{\ell}} - T^{\ell}_{ik, \,\overline{j}} = \big( R_{i\overline{j}k\overline{\ell}} + R_{k\overline{\ell}i\overline{j}}\big)  -  \big( R_{k\overline{j}i\overline{\ell}} + R_{i\overline{\ell}k\overline{j}}\big) = \sum_{r=1}^n T^r_{ik} \overline{T^r_{j\ell}}\,,$$
where the last equality is by Lemma 1. This establishes (\ref{eq:TT}). Again by Lemma 2 we have
$$ R_{\ell \overline{k}j\overline{i}} - R_{j\overline{k}\ell \overline{i}} = T^{i}_{j\ell, \,\overline{k}}. $$
Taking complex conjugate on both sides, we get
$$ R_{k\overline{\ell}i\overline{j}} - R_{k\overline{j}i\overline{\ell}} = \overline{T^{i}_{j\ell, \,\overline{k}}}. $$
Add this to the original formula of Lemma 2, we get
\begin{equation}
R_{k\overline{\ell}i\overline{j}} - R_{i\overline{j}k\overline{\ell}} \ = \ T^{\ell}_{ik, \,\overline{j}} + \overline{T^{i}_{j\ell, \,\overline{k}}}. \ \ \ \ \ \ \ \ \  \label{eq:17}
\end{equation}
On the other hand, by adding (\ref{Rhat}) with the formula of Lemma 1, we get
\begin{equation}
R_{k\overline{\ell}i\overline{j}} + R_{i\overline{j}k\overline{\ell}} \, = \, c\,\big( \delta_{ij}\delta_{k\ell} + \delta_{i\ell}\delta_{kj} \big) + \frac{1}{2} \sum_{r=1}^n T^{r}_{ik}  \overline{T^{r}_{j\ell}}.  \label{eq:18}
\end{equation}
Subtract (\ref{eq:17}) from (\ref{eq:18}) and then divide by $2$, we obtain
$$ R_{i\overline{j}k\overline{\ell}} = \frac{c}{2} \,\big( \delta_{ij}\delta_{k\ell} + \delta_{i\ell}\delta_{kj} \big) + \frac{1}{4} \sum_{r=1}^n T^{r}_{ik}  \overline{T^{r}_{j\ell}} -\frac{1}{2}  \big( T^{\ell}_{ik, \,\overline{j}} + \overline{T^{i}_{j\ell, \,\overline{k}}} \big) .$$
Now by (\ref{eq:TT}), we have
$$ \overline{T^{i}_{j\ell, \,\overline{k}}} - \overline{T^{k}_{j\ell, \,\overline{i}}} = \sum_r \overline{T^{r}_{j\ell }  \overline{T^{r}_{ik}} } = \sum_r T^{r}_{ik}  \overline{T^{r}_{j\ell}}.$$
Plug this into the above expression for $R$ we get (\ref{eq:R}). The trace taking part is clear by the definitions, and in the integral equality we used Lemma 3. This completes the proof of Theorem 2.
\end{proof}

At this point we do not know how to deduce the K\"ahlerness of the metric for compact pluriclosed manifold with constant holomorphic sectional curvature. Theorem 2 nonetheless gives strong pointwise restrictions on such manifolds. Hopefully, by utilizing the compactness assumption more, one could eventually prove Conjecture 1 for all pluriclosed manifolds.

\vspace{0.3cm}

\section{Proof of Theorem 1 }

In this section, we will restrict ourselves to a special type of pluriclosed manifolds, the so-called Strominger K\"ahler-like (or SKL for short) manifolds. Let $(M^n,g)$ be a Hermitian manifold. Denote by $\nabla^s$ its Strominger connection (also known as Bismut connection in some literature). It is the unique connection that is Hermitian (i.e., $\nabla^sg=0$, $\nabla^sJ=0$) and with totally skew-symmetric torsion:
$$ \langle T^s(x,y) , z\rangle = - \langle T^s(x,z) , y\rangle  $$
for any tangent vector $x$, $y$, $z$ on $M$. Here $T^s$ is the torsion tensor of $\nabla^s$ defined by
$$ T^s(x,y)=\nabla^s_xy - \nabla^s_yx - [x,y].$$
Let $e$ be a local unitary frame with dual coframe $\varphi$ as before, then it is well known that the Strominger connection is given by
\begin{equation}
\nabla^s e_i - \nabla e_i = \sum_j \big( \sum_k \big( T^j_{ik} \varphi_k - \overline{T^i_{jk}} \overline{\varphi}_k  \big)  \big) e_j   \label{eq:nablaminusnabla}
\end{equation}
See for example \cite[Lemma 2]{WYZ}, where $\nabla^s$ is denoted as $\nabla^b$ (note that our $T^j_{ik}$ here is twice as much as that in there). As a consequence, we get (see (22) in \cite{WYZ}) the components of $T^s$ as
\begin{equation}
T^s(e_i,e_j) = - \sum_k T^k_{ij}e_k, \ \ \ \ T^s(e_i, \overline{e}_j) = \sum_k \big( T^j_{ik} \overline{e}_k - \overline{T^i_{jk}} \,e_k \big)
\end{equation}

Let us denote by $T^j_{ik\mid \overline{\ell}}$ the components of the covariant derivative with respect to the Strominger connection $\nabla^s$. Then by (\ref{eq:nablaminusnabla}) we get
\begin{equation}
T^j_{ik\mid \overline{\ell}} - T^j_{ik, \overline{\ell}} = \sum_r \big( T^j_{kr} \overline{ T^i_{\ell r}} -  T^j_{ir} \overline{ T^k_{\ell r}} - T^r_{ik} \overline{ T^r_{j\ell }}  \big)   \label{eq:TminusT}
\end{equation}
for any $1\leq i,j,k,\ell\leq n$. Denote by $R^s$ the curvature tensor of $\nabla^s$. We have

\begin{lemma}
Given any Hermitian manifold $(M^n,g)$, then under any local unitary frame $e$,
\begin{equation}
R^s_{i\overline{j}k\overline{\ell}} - R_{i\overline{j}k\overline{\ell}} \ =  \ T^{\ell}_{ik, \,\overline{j}} + \overline{T^{k}_{j\ell, \,\overline{i}}} +  \sum_r \big(  T^{r}_{ik}  \overline{T^{r}_{j\ell}} -   T^{\ell}_{ir}  \overline{T^{k}_{jr}} \big)   \label{eq:RminusR}
\end{equation}
\end{lemma}

\begin{proof}
Let us define the linear operator $\gamma $ by letting
$$ \gamma_{ij}= \gamma'_{ij} - \overline{\gamma'_{ji} }, \ \ \ \ \ \gamma'_{ij} = \sum_k T^j_{ik} \varphi_k $$
where $e$ is any local unitary frame with $\varphi$ its dual coframe. Note that this $\gamma$ is twice of the $\gamma$ in \cite{YZ}. Denote by $\theta^s$, $\Theta^s$ the matrix of connection and curvature of $\nabla^s$ under the frame $e$. Then we always have $\theta^s=\theta + \gamma$. For a fixed point $p\in M$, we may choose $e$ so that $\theta|_p=0$. This leads to
$$ \Theta^s = d\theta^s - \theta^s \wedge \theta^s = \Theta + d\gamma - \gamma \wedge \gamma .$$
Taking the $(1,1)$-part of the entry forms, we get
$$ (\Theta^s)^{(1,1)}_{k\ell}-\Theta_{k\ell} \ = \ \overline{\partial}\gamma'_{k\ell} - \partial \overline{\gamma'_{\ell k}} + \sum_r \big( \gamma'_{kr} \overline{\gamma'_{\ell r} } + \overline{\gamma'_{rk } }  \gamma'_{r \ell } \big) $$
at the point $p$. Note that at the point $p$, since $\theta|_p=0$, the structure equation gives $\overline{\partial }\varphi =0$ at $p$, while  $\partial \varphi_i = \frac{1}{2}\sum_{j,k=1}^n T^i_{jk} \varphi_j\wedge \varphi_k$ at $p$. So by looking at the coefficients in front of the $\varphi_i \wedge \overline{\varphi}_j$ term in the above identity, we get
$$ R^s_{i\overline{j}k\overline{\ell}} - R_{i\overline{j}k\overline{\ell}} \ =  \ T^{\ell}_{ik, \,\overline{j}} + \overline{T^{k}_{j\ell, \,\overline{i}}} +  \sum_r \big(  T^{r}_{ik}  \overline{T^{r}_{j\ell}} -   T^{\ell}_{ir}  \overline{T^{k}_{jr}} \big) .$$
This completes the proof of the lemma.
\end{proof}

Now let us suppose that $(M^n,g)$ is a SKL manifold. By \cite{ZhaoZ1}, we know that $g$ is pluriclosed, and $\nabla^sT^s=0$. Clearly, the latter is equivalent to $\nabla^sT=0$. Thus by (\ref{eq:TminusT}), we know that for any SKL manifold, the Chern covariant derivative of the torsion is given by a quadratic form of the torsion:
 \begin{equation}
 T^j_{ik, \overline{\ell}} = \sum_r \big( - T^j_{kr} \overline{ T^i_{\ell r}} +  T^j_{ir} \overline{ T^k_{\ell r}} + T^r_{ik} \overline{ T^r_{j\ell }}  \big)   \label{eq:TminusT2}
\end{equation}
for any $1\leq i,j,k,\ell\leq n$. Using this formula to replace the two derivative terms in Lemma 5, we get

\begin{lemma}
If $(M^n,g)$ is a SKL manifold, then under any local unitary frame $e$ it holds
\begin{equation}
R^s_{i\overline{j}k\overline{\ell}} - R_{i\overline{j}k\overline{\ell}} \ =  \  \sum_r \big( T^{\ell}_{ir}  \overline{T^{k}_{jr}} -  T^{r}_{ik}  \overline{T^{r}_{j\ell}} -   T^{j}_{ir}  \overline{T^{k}_{\ell r}}-   T^{\ell}_{kr}  \overline{T^{i}_{jr}} \big)   \label{eq:RminusR2}
\end{equation}
\end{lemma}

Let us denote by $\widehat{R}^s$ the symmetrization of $R^s$, namely,
\begin{equation}
\widehat{R}^s_{i\overline{j}k\overline{\ell}} = \frac{1}{4} \big( R^s_{i\overline{j}k\overline{\ell}} + R^s_{k\overline{j}i\overline{\ell}} +  R^s_{i\overline{\ell}k\overline{j}} +  R^s_{k\overline{\ell}i\overline{j}} \big)
\end{equation}
Taking the symmetrization of the formula in Lemma 6, and using the fact that $R^s=\widehat{R}^s$ in the SKL case, a straight forward computation leads to the following
\begin{lemma}
For any SKL manifold, then under any unitary frame it holds that
\begin{equation}
R^s_{i\overline{j}k\overline{\ell}} \ = \ \widehat{R}_{i\overline{j}k\overline{\ell}} -  \sum_r \big( T^{j}_{ir}  \overline{T^{k}_{\ell r}} + T^{\ell}_{ir}  \overline{T^{k}_{jr}} +     T^{j}_{kr}  \overline{T^{i}_{\ell r}} +   T^{\ell}_{kr}  \overline{T^{i}_{jr}} \big)   \label{eq:Rs}
\end{equation}
\end{lemma}

Next, we will need the following properties of any non-K\"ahler SKL manifold $(M^n,g)$. We refer the readers to \cite{ZhaoZ1} and \cite{YZZ} for their proofs.
\begin{lemma}
Let $(M^n,g)$ be a  non-K\"ahler SKL manifold. Then for any $p\in M$, there exists a local unitary frame $e$ in a neighborhood of $p$ such that $\eta_n=\lambda$ where $\lambda$ is a positive constant, $\eta_1=\cdots = \eta_{n-1}=0$, and
$$ T^n_{\ast \ast }=0, \ \ \ T^j_{in}=\delta_{ij} a_i, \ \ \ R^s_{i\overline{j}k\overline{n}}=0 $$
for any indices $i$, $j$, $k$, where $a_i$ are constants with $a_n=0$ and $a_1 +\cdots + a_{n-1}=\lambda$.
\end{lemma}

Such a local frame $e$ is called an {\em admissible frame}. Now we are finally ready to prove Theorem 1 stated in the introduction section.

\vspace{0.3cm}

\begin{proof}[\textbf{Proof of Theorem 1:}] Let $(M^n,g)$ be a SKL manifold with constant holomorphic sectional curvature, namely, $H=c$. We want to show that $M^n$ must be K\"ahler, hence a complex space form. Assume that $g$ is not K\"ahler, then  locally around any $p\in M$ there always exists an admissible unitary frame $e$ by Lemma 8, and we have $R^s_{n\overline{n} n\overline{n}} =0$. Since $T^n_{\ast \ast }=0$,  Lemma 7 leads to
$$ R^s_{n\overline{n} n\overline{n}} = \widehat{R}_{n\overline{n} n\overline{n}} $$
Thus $c=0$, hence $\widehat{R}=0$. Again by Lemma 7, we have
\begin{equation*}
R^s_{i\overline{j}k\overline{\ell}} \ = \  -  \sum_r \big( T^{j}_{ir}  \overline{T^{k}_{\ell r}} + T^{\ell}_{ir}  \overline{T^{k}_{jr}} +     T^{j}_{kr}  \overline{T^{i}_{\ell r}} +   T^{\ell}_{kr}  \overline{T^{i}_{jr}} \big)   \label{eq:Rs2}
\end{equation*}
for any $1\leq i,j,k,\ell \leq n$. Letting $\ell =n$ in the above identity, the second and fourth terms on the right hand side vanish, while the diagonal property of $T^{\ast }_{\ast n}$ gives us
\begin{equation*}
0 \ = \ R^s_{i\overline{j}k\overline{n}} \ = \    (\overline{a}_k - \overline{a}_i)  T^{j}_{ik}
\end{equation*}
for any $1\leq i,j,k\leq n$. Let $k=n$ and $i=j$, we get $|a_i|^2=0$, hence $a_i=0$ for each $i$. But this will violate the property $a_1 + \cdots + a_{n-1}=\lambda >0$, hence the metric $g$ must be K\"ahler to begin with, and we have completed the proof of Theorem 1. \end{proof}

As a final remark, we notice that in Theorem 1, one actually does not need to assume that $M^n$ is compact,  due to the strong properties of non-K\"ahler SKL manifolds. But for general pluriclosed manifold with constant holomorphic sectional curvature, we believe that the compactness assumption is necessary.

\vspace{0.5cm}

\noindent\textbf{Acknowledgments.} {The second named author would like to thank mathematicians  Haojie Chen, Xiaolan Nie, Kai Tang, Bo Yang, Xiaokui Yang, and Quanting Zhao for their interests and/or helpful discussions.}


\end{document}